\documentclass[letterpaper,12pt,reqno]{amsart}
\usepackage{graphicx,amsmath,amssymb, amsthm}

\usepackage{srcltx}
\usepackage[latin1]{inputenc}
\usepackage[T1]{fontenc}

\usepackage{fullpage}

\newtheorem{X}{X}[section]
\newtheorem{corollary}[X]{Corollary}

\newtheorem{lemma}[X]{Lemma}
\newtheorem{proposition}[X]{Proposition}
\newtheorem{theorem}[X]{Theorem}

\theoremstyle{definition}
\newtheorem{remark}[X]{Remark}

\newcommand{\be}{\overrightarrow{\beta}}

\renewcommand{\be}{\begin{equation}}
\newcommand{\ee}{\end{equation}}
\newcommand\bea{\begin{eqnarray}}
\newcommand\eea{\end{eqnarray}}
\newcommand\bi{\begin{itemize}}
\newcommand\ei{\end{itemize}}
\newcommand\ben{\begin{enumerate}}
\newcommand\een{\end{enumerate}}
\newcommand\bc{\begin{center}}
\newcommand\ec{\end{center}}
\newcommand\ba{\begin{array}}
\newcommand\ea{\end{array}}

\title{On Vaughan's approximation: The first moment}
\author{Daniel Fiorilli}
\address{Département de mathématiques et de statistique, Université d'Ottawa, 
585 King Edward, Ottawa, Ontario, K1N 6N5, Canada}
\email{daniel.fiorilli@uottawa.ca}
\date{\today}

\begin{document}

\begin{abstract}

We investigate the first moment of the difference between $\psi(x;q,a)$ and Vaughan's approximation, in a certain range of $q$. We show that this last approximation is significantly more precise than the classical $x/\phi(q)$, and that it captures the discrepancies of the distribution of primes in arithmetic progressions found in an earlier paper of the author. 
\end{abstract}
\maketitle

\section{Introduction}

The moments of the error term in the prime number theorem in arithmetic progressions
are a central object of study and have been extensively studied in the literature. 
Upper bounds for the first moment, which apply to the Titchmarsh divisor problem, were obtained by Fouvry \cite{Fo}, Bombieri, Friedlander and Iwaniec \cite{BFI}, Friedlander and Granville \cite{FG} and Friedlander, Granville, Hildebrand and Maier \cite{FGHM}.

\begin{theorem}[{\cite[Theorem 1]{FG}, \cite[Proposition 2.1]{FGHM}}]
\label{theorem FG}
Let $0<\lambda<1/4$, $A>0$ be given. Then uniformly for $0<|a|<x^{\lambda}$, $2\leq Q \leq x/3$ we have 
\begin{equation}
\label{eq:FG}
\sum_{\substack{Q<q\leq 2Q \\ (q,a)=1}} \left( \psi(x;q,a)-\frac {\psi(x)}{\phi(q)} \right) \ll_{\lambda,A} 2^{\omega(a)} Q \log(x/Q) + \frac x{(\log x)^A} +Q\log |a|.
\end{equation}
\end{theorem}
\noindent 
These results are based on the dispersion method and deep estimates on sums of Kloosterman sums \cite{DI}, and generalize to other arithmetic sequences such as friable integers in arithmetic progressions \cite{FT,Dr}. 

In \cite{Fi}, the author showed that in some cases it is possible to obtain an asymptotic formula for the quantity on the left hand side of \eqref{eq:FG}.
  
\begin{theorem}[{\cite[Theorem 1.1]{Fi}}]
\label{theorem Fi}
Fix an integer $a\neq 0$, a positive real number $B$ and $\epsilon>0$. Then, for $M=M(x) \leq (\log x)^B$, one has
\begin{equation}
\label{enonce premier thm}
\frac1{\frac{\phi(a)}{a}\frac{x}{M}  }\sum_{\substack{q\leq \frac{x} {M}  \\(q,a)=1}} \left( \psi(x;q,a)-\Lambda(a)-\frac{\psi(x)}{\phi(q)}\right) = \mu(a,M)+O_{a,\epsilon,B} \left(\frac{1}{M^{\frac {205}{538}-\epsilon}}\right)
\end{equation}
with 
\begin{equation}
\mu(a,M):=  \begin{cases}
-\frac 12 \log M - C_0 &\text{ if } a=\pm 1\\
			-\frac 12 \log p &\text{ if } a=\pm p^e \\
			 0 &\text{ otherwise,}
                    \end{cases}
                    \label{equation definition of mu}
\end{equation} 
where 
\begin{equation*} C_0:= \frac 12 \left(\log 2\pi + \gamma +\sum_p \frac {\log p}{p(p-1)}+1 \right).
\end{equation*}
\end{theorem}

\begin{remark}
The exponent $205/538$ in Theorem \ref{theorem Fi}, which comes from Huxley's subconvexity estimate \cite{Hu}, can be improved to $171/448$ using Bourgain's recent work \cite{Bo}.
\end{remark}

\begin{remark}
In Theorem \ref{theorem Fi} we have excluded the first term $n=a$ of the arithmetic progression $a \bmod q$; we will keep doing so and use the notation
$$ \psi^*(x;q,a) := \sum_{\substack{ n\leq x \\ n\equiv a \bmod q \\ n> a}} \Lambda(n).  $$
The reason we do this is because the term $\Lambda(a)$ can have a significant contribution in this context, and this contribution is trivial to control. 
\end{remark}

One can interpret Theorem \ref{theorem Fi} by saying that the discrepancy of the distribution of primes in the different arithmetic progressions $a \bmod q$ (with $(a,q)=1$) is negative for $a$ having at most one prime factor, and is zero otherwise. One could ask whether there exists an approximation to $\psi(x;q,a)$, superior to $\psi(x)/\phi(q)$, which has the same discrepancies as   $\psi(x;q,a)$. In the present paper we will show that Vaughan's approximation has this property.

Vaughan introduced the following approximation
to $\psi(x;q,a)$, which depends on a parameter $R\geq 1$:
$$ \rho_R(x;q,a):= \sum_{\substack{ n\leq x \\ n\equiv a \bmod q  }} F_R(n),$$ where
$$ F_R(n):=\sum_{r\leq R} \frac{\mu(r)}{\phi(r)} \sum_{\substack{ 1\leq b\leq r \\ (b,r)=1}} e(bn/r)= \sum_{r\leq R} \frac{\mu^2(r)\mu((r,n))\phi((r,n))}{\phi(r)}. $$
The function $F_R(n)$ was motivated by the Hardy-Littlewood method, in order to remove the contribution of the major arcs.
Remarkably, Vaughan showed that the second \cite[Corollary 4.1]{V1} and third \cite[Theorem 8]{V2} moments of $\psi(x;q,a)-\rho_R(x;q,a)$, averaged over $q\leq x/M$ with $M,R\leq (\log x)^A$, are smaller than those of $\psi(x;q,a)-\psi(x)/\phi(q)$ when $R$ is larger than $M$ (and the implied error terms are sharper than \cite[Theorem 1.1]{GV} and \cite[Theorems 1,2]{Ho8}).

Our first result shows that Vaughan's approximation has the properties described earlier, that is it captures the discrepancies of $\psi(x,q;a)$ in the arithmetic progressions $a \bmod q$ observed in Theorem \ref{theorem Fi}.
As we did with $\psi(x;q,a)$ above, we exclude the first term of the arithmetic progression $a \bmod q$:
\begin{equation}
\rho_R^*(x;q,a):= \sum_{\substack{ n\leq x \\ n\equiv a \bmod q \\ n > a }} F_R(n).
\label{equation definition rho*}
\end{equation} 
In what follows, $R$ should be thought as a fixed power of $\log x$, however it can be even smaller when looking at moduli $q$ very close to $x$.

\begin{theorem}
\label{theorem main dyadique}
Fix $A,B \geq 1$. 

(i) Uniformly for $0<|a|\leq x/(\log x)^{A+B+1}$, $1\leq M\leq (\log x)^A$ and $2M \leq R \leq x^{\frac 12}$ we have
\begin{equation} \label{equation main dyadique} \frac 1{x/2M}\sum_{\substack{ \frac x{2M}<q\leq \frac xM }} \left( \psi^*(x;q,a)-\rho_R^*(x;q,a) \right) \ll_{A,B} \frac{1}{(\log x)^B}. \end{equation}

(ii) If in addition $2|a|M \leq R$, then restricting the sum over moduli coprime to $a$,
\begin{equation} \label{equation main dyadique coprime} \frac 1{\frac{\phi(a)}a \frac x{2M}}\sum_{\substack{ \frac x{2M}<q\leq \frac xM \\ (q,a)=1 }} \left( \psi^*(x;q,a)-\rho_R^*(x;q,a) \right) \ll_{A,B} \frac{1}{(\log x)^B}. \end{equation}

\end{theorem} 
Comparing with (a dyadic version of) Theorem \ref{theorem Fi}, we deduce that $\rho_R^*(x;q,a)$ is a much better approximation to $\psi^*(x;q,a)$ than $\psi(x)/\phi(q)$, on average over $q\asymp x/M$. Indeed, for $M\rightarrow \infty$, the right hand sides of \eqref{equation main dyadique} and \eqref{equation main dyadique coprime} are $\ll_K M^{-K}$ for any $K\geq 1$, and are independent of both $a$ and $R$. They are also much smaller than \eqref{enonce premier thm} for fixed values of $M$.

Let us briefly explain why it is possible to obtain such an error term in Theorem \ref{theorem main dyadique}. In Theorem \ref{theorem Fi}, the error term comes from the cancellation of main terms in sums of a certain multiplicative function. In the corresponding situation for Theorem \ref{theorem main dyadique}, we have cancellation of the whole sums of the implied multiplicative function, rather than just the main terms (see Lemmas \ref{lemma big values of q} (ii) and \ref{lemma average of psi}).

\begin{remark}
In Theorem \ref{theorem main dyadique} (i), we sum over all moduli $q$, not just those coprime to $a$. The reason we do this is that when $(q,a)>1$, both $\psi^*(x;q,a)$ and $\rho_R^*(x;q,a)$ are small.
Note however that (ii) is not a direct consequence of (i), since contrary to $\psi^*(x;q,a)$, it is not trivial to handle $\rho_R^*(x;q,a)$ when $(q,a)>1$ (see Section \ref{section not coprime} for more details).
\end{remark}

Things are quite different when averaging over the whole range $q\leq x/M$. Indeed in this case we obtain non-negligible lower-order terms. This result seems to indicate that Vaughan's approximation is better for larger values of $q$ than for more moderate ones.

\begin{theorem}
\label{theorem main full range}
Fix $A,B \geq 1$, and $a\neq 0$. 

(i) Uniformly for $M\leq (\log x)^A$ and $1\leq M \leq R \leq x^{\frac 12}/(\log x)^{A+B}$ we have
\begin{multline}
\frac 1{ x/M}\sum_{\substack{ q\leq \frac xM }} \left( \psi^*(x;q,a)-\rho_R^*(x;q,a) \right) =  \epsilon_{a=\pm 1}  \frac{M}R\left(\log \frac x{R^2} +2\gamma-3\right)  \\
+O_{a}\bigg(\frac{M\log x}{R^{\frac 32} \exp \big(c \frac{(\log R)^{\frac 35}}{(\log\log R)^{\frac 15}} \big)} \bigg) +O_{a,A,B}\left(  \frac{1}{(\log x)^B} \right),
\label{equation full range} 
\end{multline}
where $\epsilon_{a=\pm 1}$ equals $1$ if $a=\pm 1$, and is zero otherwise.

(ii) Under the additional condition $|a|M \leq R$, we have that
\begin{multline}
\frac 1{\frac{\phi(|a|)}{|a|} \frac xM}\sum_{\substack{ q\leq \frac xM \\ (q,a)=1 }} \left( \psi^*(x;q,a)-\rho_R^*(x;q,a) \right) =    \frac{\phi(|a|)}{|a|} \frac{M}R\Big(\log \frac x{R^2} +2\gamma-3+\sum_{p\mid a} \frac{p+1}{p-1}\log p\Big)  \\
+O_{a}\bigg(\frac{M\log x}{R^{\frac 32} \exp \big(c \frac{(\log R)^{\frac 35}}{(\log\log R)^{\frac 15}} \big)} \bigg) +O_{a,A,B}\left(  \frac{1}{(\log x)^B} \right).
\label{equation full range coprime} 
\end{multline}
In both of these statements, $c$ is a positive absolute constant.
\end{theorem} 

\begin{remark}
Fixing $a\notin \{0,\pm1\}$ and comparing \eqref{equation full range} and \eqref{equation full range coprime}, we see that contrary to the situation in Theorem \ref{theorem main dyadique}, $\rho_R^*(x;q,a)$ has a non-trivial contribution when $(q,a)>1$. This indicates once more that Vaughan's approximation is more precise for larger values of $q$. We will expand on this remark in Section \ref{section not coprime}. 
\end{remark}

\begin{remark}
Taking $M=1$ in Theorem \ref{theorem main full range} (i)\footnote{Note that this theorem itself is based on the results of \cite{Fo, BFI}.} and applying Lemmas \ref{lemma average of rho q<x} and \ref{lemma sum over r with gcd} we recover the known estimate for the Titchmarsh divisor problem \cite{Fo, BFI}. Drappeau recently established \cite{Dr2} that the error term in this problem depends on the existence of Landau-Siegel zeros.
\end{remark}

Comparing Theorems \ref{theorem Fi} and \ref{theorem main full range}, we see that $\rho_R^*(x;q,a)$ necessarily has the same discrepancies in arithmetic progressions as $\psi^*(x;q,a)$, when averaged over $q\leq x/M$ with $M\leq (\log x)^{O(1)}$. We will show that these discrepancies persist for $M$ as large as $x^{\frac 12-\epsilon}/R$, as long as $M\leq R$.

\begin{proposition}
\label{proposition discrepancies rho}
Fix $\epsilon>0$ and $a\neq 0$.
Uniformly for $1\leq |a|M\leq R \leq x^{\frac 12}$ we have
\begin{multline}
\frac 1{\frac{\phi(|a|)}{|a|}\frac xM}\sum_{\substack{ q\leq \frac xM \\ (q,a)=1 }} \left( \rho_R^*(x;q,a)-\frac x{\phi(q)} \right)= \mu(a,M) \\-\frac{\phi(|a|)}{|a|} \frac{M}R\Big(\log \frac x{R^2} +2\gamma-3+\sum_{p\mid a} \frac{p+1}{p-1}\log p\Big) +O_{a}\bigg( \frac {1}{M^{\frac{171}{448}-\epsilon}}+\frac {M\log x}{R^{\frac 32} \exp \big(c \frac{(\log R)^{\frac 35}}{(\log\log R)^{\frac 15}} \big)} +\frac{RM}{x^{\frac 12}}\bigg),
\label{equation proposition discrepancy}
\end{multline}
where $\mu(a,M)$ is defined in \eqref{equation definition of mu}.
\end{proposition}
Note that by Lemma \ref{lemma discrepancy}, the quantity $\rho_R^*(x;q,a)-\frac x{\phi(q)}$ approximately equals the discrepancy (with signs) of the distribution of $F_R(n)$ in the arithmetic progressions $a \bmod q$ with $(a,q)=1$. 

\begin{remark}
Combining either \eqref{eq:lemma divisor switch} or \eqref{equation average rho coprime without condition on M} with the formula
$$ \sum_{n\leq x} \left( \frac 1n-\frac 1x \right)= \log x+\gamma-1+\frac 1{2x} +O\left( \frac 1{x^2} \right) \hspace{1.5cm} (x\in \mathbb R_{\geq 1}), $$
one can estimate the quantities in Theorems \ref{theorem main dyadique}, \ref{theorem main full range} and Proposition \ref{proposition discrepancies rho} in the range $R < M \leq R^{1+\delta}$, for some $\delta>0$. 
The resulting bounds are weaker than in the case $R\geq M$, and thus we decided not to pursue this further.
\end{remark}

\section{The dyadic average}

Let us first recall two results of \cite{V1}. The proofs of these results are contained\footnote{In Lemma \ref{lemma:additive character sum2} we have used the identity $\mu(r)\mu(r/(r,a))/\phi(r/(r,a)) = \mu^2(r)\mu((r,a)) \phi((r,a))/\phi(r) $.} in that of \cite[Theorem 1]{V1} and will therefore be omitted. 

\begin{lemma}
\label{lemma:additive character sum}
Assume that $a$, $r$ and $s$ are integers with $r,s \geq 1$. We have for $a\leq  y \leq x $ with $y\geq 0$ that
\begin{equation}
\label{eq:additive character lemma}
\sum_{\substack{ 1\leq b\leq r \\ (b,r)=1}}\sum_{\substack{ y< n\leq x \\ n\equiv a \bmod s  }} e(bn/r)= \delta_{r \mid s} \frac{x-y}s \frac{\mu(r/(r,a))\phi(r)}{\phi(r/(r,a))}  +O\left(r\log r\right),
\end{equation}
where $\delta_{r \mid s}$ equals $1$ when $r \mid s$, and $0$ otherwise.
\end{lemma}
\begin{lemma}
\label{lemma:additive character sum2}
Let $a \in \mathbb Z$ and $s\in \mathbb Z_{\geq 1}$. If $a\leq  y \leq x $ and $y\geq 0$, then
\begin{equation}
\label{eq:additive character lemma2}
\sum_{\substack{ y< n\leq x \\ n\equiv a \bmod s  }} F_R(n)= \frac{x-y}s \sum_{\substack{ r\leq R \\ r\mid s}} \frac{\mu^2(r)\mu((r,a)) \phi((r,a))}{\phi(r)}  +O(R).
\end{equation}

\end{lemma}

\begin{remark}
Lemma \ref{lemma:additive character sum2} implies that
\begin{align}
\rho_R^*(x;q,a) =\frac xq \sum_{\substack{ r\leq R \\ r\mid q}} \frac{\mu^2(r)\mu((r,a))\phi((r,a))}{\phi(r)} + O(R). 
\label{equation first estimate for rho}
\end{align} 
This expression precise when $q$ is small compared to $x$ (c.f. \cite[Theorem 1, Corollaries 1.1-1.2]{V1}); for example when $q\leq R$ it takes the form  
$$ \rho_R^*(x;q,a) = \delta_{(q,a)=1}\frac x{\phi(q)}  + O(R). $$
However, \eqref{equation first estimate for rho} is not accurate when $q$ is close to $x$. Nevertheless we will see by a different approach (see for instance the proof of Theorem \ref{theorem main dyadique} (i)) that on average over large $q$, $\rho_R^*(x;q,a)$ is much closer to $\psi^*(x;q,a)$ than to $\delta_{(q,a)=1} x/\phi(q)$.
\end{remark}

We will average $\psi(x;q,a)$ and $\rho_R^*(x;q,a)$ over $q$ close to $x$ separately. We begin with $\rho_R^*(x;q,a)$.

\begin{lemma}
\label{lemma big values of q}
(i) For $0<|a| < x/N$ and $1\leq N,R\leq x$, we have 

\begin{multline} \sum_{\substack{ \frac xN <q\leq x }} \rho_R^*(x;q,a) = x\sum_{\substack{s \leq  N}} \frac 1s \left( 1- \frac sN \right)\sum_{\substack{r\leq R \\ r\mid s}} \frac{\mu^2(r)\mu((r,a))\phi((r,a))}{\phi(r)}  
 \label{eq:lemma divisor switch}
+O(  
RN + |a|(\log N)^2).
\end{multline}
(ii) Under the additional condition $N\leq R$ we have
\begin{align}\sum_{\substack{ \frac xN <q\leq x }} \rho_R^*(x;q,a) = 
x \sum_{\substack{ s\leq N \\ (s,a)=1}} \frac 1{\phi(s)} \left( 1-\frac sN\right)
+O\left(  
RN + |a|\log N\right).
\label{equation corollary sum of rho}
\end{align}

\end{lemma}

\begin{proof}
We rewrite the conditions $n\equiv a \bmod q; n> a; x/N <q\leq x$ as $n=a+qs,$ with $1\leq s< N-ax/N$ and $a+sx/N< n\leq x$. We obtain that
\begin{align*}
\sum_{\substack{ \frac xN <q\leq x }} \rho_R^*(x;q,a) &=
\sum_{\substack{ 1\leq s < N -\frac{aN}x}}\sum_{\substack{ a+\frac{sx}N< n\leq x \\ n\equiv a \bmod s  }} F_R(n).
\end{align*}

Applying Lemma \ref{lemma:additive character sum2} with $y=a+sx/N>0$, we see that this expression equals
\begin{align}
&x\sum_{\substack{ 1\leq s < N -\frac{aN}x }} \frac 1s \left( 1- \frac sN -\frac ax\right)\sum_{\substack{r\leq R \\ r\mid s}} \frac{\mu^2(r)\mu((r,a))\phi((r,a))}{\phi(r)}+O(RN) \label{eq:RN} \\
&=x\sum_{\substack{ 1\leq s < N-\frac{aN}x}} \frac 1s  \left( 1- \frac sN \right) \sum_{\substack{r\leq R \\ r\mid s}} \frac{\mu^2(r)\mu((r,a))\phi((r,a))}{\phi(r)}+O(RN+|a| (\log N)^2) \notag
\\&= x\sum_{\substack{ 1\leq s \leq N }} \frac 1s \left( 1- \frac sN \right) \sum_{\substack{r\leq R \\ r\mid s}} \frac{\mu^2(r)\mu((r,a))\phi((r,a))}{\phi(r)} +O\left( RN+|a|(\log N)^2 \right),\notag
\end{align}
since for $s\in (N-|a|N/x,N+|a|N/x)$ we have that $ |1-s/N|< |a|/x$. The estimate \eqref{eq:lemma divisor switch} follows.

To establish \eqref{equation corollary sum of rho} we come back to \eqref{eq:RN}. Under the condition $N\leq R$, we have that \eqref{eq:RN} equals (the second error term in the following expression is only present in the case $a<0$)
\begin{align*} &x\sum_{\substack{ 1\leq s < N -\frac{aN}x }} \frac 1s\left( 1- \frac sN -\frac ax\right) \sum_{\substack{ r\mid s}} \frac{\mu^2(r)\mu((r,a))\phi((r,a))}{\phi(r)} +O(RN) \\
&\hspace{2cm} +O\Big( x\sum_{\substack{ N <  s < N -\frac{aN}x }} \frac 1{N}\cdot \frac{|a|}x \sum_{\substack{ r\mid s \\ \frac s2 \leq  r \leq s }} 1 \Big) \\
&=x\sum_{\substack{ 1\leq s < N -\frac{aN}x \\ (s,a)=1}} \frac 1{\phi(s)} \left( 1- \frac sN -\frac ax\right) +O(|a|+RN),
\end{align*}
by multiplicativity. Note that if $(s,a)>1$, then 
$$ \sum_{\substack{r\mid s}} \frac{\mu^2(r)\mu((r,a))\phi((r,a))}{\phi(r)} = \prod_{\substack{p\mid s \\ p\nmid a}} \left( 1+\frac 1{p-1} \right) \prod_{\substack{ p\mid s \\ p\mid a}} (1-1) = 0. $$
The proof follows.

\end{proof}

We now average $\psi^*(x;q,a)$ over $q$ close to $x$.

\begin{lemma}
\label{lemma average of psi}
Fix $A,B\geq 1$. In the range $1\leq N\leq (\log x)^A$ and for $0<|a| < x/N$ we have 

\begin{equation} \sum_{\substack{ \frac xN <q\leq x }} \psi^*(x;q,a) = x\sum_{\substack{ s \leq N \\ (s,a)=1}} \frac{1}{\phi(s)} \left( 1-\frac sN\right)
+O_{A,B}\left( \frac x{(\log x)^{B}} \right)+ O( |a|\log N).
\end{equation}
\end{lemma}

\begin{proof}

The proof is achieved by swapping moduli as in the proof of Lemma \ref{lemma big values of q} and applying the Siegel-Walfisz Theorem. We have
\begin{align*}
\sum_{\substack{ \frac xN <q\leq x }} \psi^*(x;q,a) &=\sum_{\substack{ 1\leq s < N -\frac{aN}x}}\sum_{\substack{ a+\frac{sx}N< n\leq x \\ n\equiv a \bmod s  }} \Lambda(n).\\
&= \sum_{\substack{ 1\leq s < N -\frac{aN}x \\ (s,a)=1}} \frac{x-(a+sx/N)}{\phi(s)} +O_{A,B}\left( \frac x{(\log x)^{B}}\right)\\
&= x\sum_{\substack{ 1\leq s \leq N \\ (s,a)=1}} \frac{1}{\phi(s)} \left( 1-\frac sN\right) +O_{A,B}\left( \frac x{(\log x)^{B}}\right)+O( |a|\log N).
\end{align*}
\end{proof}

\begin{corollary}
\label{corollary large q}
Fix $A,B \geq 1$. For $0<|a| < x/N$, $1\leq N\leq (\log x)^A$ and $N\leq R \leq x/(\log x)^{A+B}$ we have
\begin{align*}\sum_{\substack{ \frac xN <q\leq x }} (\psi^*(x;q,a) -\rho_R^*(x;q,a)) = 
O_{A,B}\left(
\frac x{(\log x)^B} \right)+O(|a|\log N ).
\end{align*}

\end{corollary}
\begin{proof}
Combine Lemmas \ref{lemma big values of q} (ii) and \ref{lemma average of psi}. Note that the main terms in these estimates are identical. 
\end{proof}
We are ready to prove Theorem \ref{theorem main dyadique} (i).

\begin{proof}[Proof of Theorem \ref{theorem main dyadique} (i)]
Take $N=2M$ and $N=M$ in Corollary \ref{corollary large q}, and subtract the resulting expressions. 
\end{proof}

\section{Averages of multiplicative functions}

In this section we give estimates on averages of multiplicative functions which will be needed in Sections \ref{section all moduli} and \ref{section coprimality} to average $\rho(x;q,a)$ over the full range $q\leq x/M$. The following two constants will appear repeatedly:
$$C_1(a):=\frac{ \zeta(2)\zeta(3)}{\zeta(6)}\frac {\phi(a)}{a}\prod_{p\mid a} \left( 1-\frac 1 {p^2-p+1}\right),$$
$$C_2(a):=C_1(a)\left(\gamma -1 - \sum_{p} \frac {\log p} {p^2-p+1}+\sum_{p\mid a} \frac{p^2\log p}{(p-1)(p^2-p+1)}\right). $$

\begin{lemma}
\label{lemma squarefree}
There exists an absolute constant $c$ such that for $x\in \mathbb R_{\geq 3}$ and $\ell \in \mathbb Z_{\geq 1}$ with $\ell\leq x^{10}$,
\begin{equation}
\sum_{\substack{n > x \\ (n,\ell)=1}} \frac {\mu^2(n)}{n^2} =  \frac 1{\zeta(2)x}\prod_{p\mid \ell} \left( 1+\frac 1{p} \right)^{-1} +O\Bigg( \frac{1} {x^{\frac 32} \exp \big(c \frac{(\log x)^{\frac 35}}{(\log\log x)^{\frac 15}} \big)}\Bigg);
\label{eq:squarefree}
\end{equation}  
\begin{equation}
\sum_{\substack{n > x \\ (n,\ell)=1}} \frac {\mu^2(n)\log n}{n^2} =  \frac {\log x+1}{\zeta(2)x}\prod_{p\mid \ell} \left( 1+\frac 1{p} \right)^{-1} +O\Bigg( \frac{1} {x^{\frac 32} \exp \big(c \frac{(\log x)^{\frac 35}}{(\log\log x)^{\frac 15}} \big)}\Bigg).
\label{eq:squarefree log}
\end{equation}  

\end{lemma}

\begin{proof}

We first record the unconditional bound on the Mertens function, which follows from the Korobov-Vinogradov zero-free region for $\zeta(s)$:
$$ \sum_{\substack{ n\leq x  }} \mu(n) \ll x \exp(-c_2(\log x)^{\frac 35}(\log\log x)^{-\frac 15}). $$
Proceeding as in \cite[Exercise 6.2.19]{MoVa} we recover the classical estimate 
$$ \sum_{n\leq x} \mu^2(n) = \frac x{\zeta(2)} +O(x^{\frac 12}\exp(-c_1 (\log x)^{\frac 35} (\log\log x)^{-\frac 15}) ). $$
Combining this with the identity\footnote{By $d\mid \ell^{\infty}$ we mean that $d$ is a positive integer such that each of its prime factors divides $\ell$.}
$$  \sum_{\substack{n\leq x \\ (n,\ell)=1}} \mu^2(n) = \sum_{d\mid \ell^{\infty }} \lambda(d) \sum_{\substack{m\leq x/d }} \mu^2(m), $$
we obtain that
$$  \sum_{\substack{n\leq x \\ (n,\ell)=1}} \mu^2(n) = \frac{x}{\zeta(2)} \prod_{p\mid \ell} \left( 1+\frac 1p \right)^{-1} +O\Big(x^{\frac 12}\exp(-c_1 (\log x)^{\frac 35} (\log\log x)^{-\frac 15}) ) \sum_{\substack{d\mid \ell^{\infty} \\ d\leq x}} 1 \Big).  $$
The sum in the error term is easily shown to be bounded by a constant times $(\log x)^2$. The estimates \eqref{eq:squarefree} and \eqref{eq:squarefree log} follows from applying summation by parts.
\end{proof}

\begin{lemma}
\label{lemma sum over r pure}
There exists an absolute constant $c>0$ such that for $R\in \mathbb R_{\geq 3}$ we have the estimates
$$
\sum_{r > R} \frac{\mu^2(r)}{r\phi(r)} = \frac 1R +O\Bigg( \frac{1} {R^{\frac 32} \exp \big(c \frac{(\log R)^{\frac 35}}{(\log\log R)^{\frac 15}} \big)}  \Bigg);
$$
$$
\sum_{r > R} \frac{\mu^2(r)\log r}{r\phi(r)} = \frac {\log R+1}R +O\Bigg( \frac{1} {R^{\frac 32} \exp \big(c \frac{(\log R)^{\frac 35}}{(\log\log R)^{\frac 15}} \big)}  \Bigg).
$$
\end{lemma}

\begin{proof}
Using the convolution identity $r/\phi(r) = \sum_{d\mid r} \mu^2(d)/\phi(d)$ and applying Lemma \ref{lemma squarefree}, we have that
\begin{align*}
\sum_{r> R} \frac{\mu^2(r)}{r\phi(r)} &= \sum_{d \geq 1} \frac{\mu^2(d)}{\phi(d)}  \sum_{\substack{m > R/d }} \frac{\mu^2(dm)}{d^2m^2} \\
&=\sum_{d \leq R^{\frac 45}} \frac{\mu^2(d)}{d^2\phi(d)}  \sum_{\substack{m > R/d \\ (m,d)=1 }} \frac{\mu^2(m)}{m^2} +O\left( \frac 1{R^{\frac 85}}\right)\\
&=\frac 1{R\zeta(2)} \sum_{d \leq  R^{\frac 45}} \frac{\mu^2(d)}{d\phi(d)}\prod_{p\mid d} \left( 1+\frac 1p \right)^{-1} +O\Bigg( \frac{1} {R^{\frac 32} \exp \big(c \frac{(\log R)^{\frac 35}}{(\log\log R)^{\frac 15}} \big)} \sum_{d \leq R^{\frac 45}} \frac{\mu^2(d)}{d^{\frac 12}\phi(d)} + \frac 1{R^{\frac 85}} \Bigg) \\
&= \frac 1{R\zeta(2)} \sum_{d \geq 1}  \frac{\mu^2(d)}{d\phi(d)}\prod_{p\mid d} \left( 1+\frac 1p \right)^{-1}+O\Bigg( \frac{1} {R^{\frac 32} \exp \big(c \frac{(\log R)^{\frac 35}}{(\log\log R)^{\frac 15}} \big)}  \Bigg).
\end{align*}  
The first result follows from a straightforward computation, and the second from a summation by parts.
\end{proof}

\begin{lemma}
\label{lemma sum over r coprime to a}
There exists an absolute constant $c>0$ such that for $a\in \mathbb Z_{\neq 0}$ and $R\in \mathbb R_{\geq 3}$ we have the estimates
$$
\sum_{\substack{r\leq R \\ (r,a)=1}} \frac{\mu^2(r)}{r\phi(r)} = \frac{|a|}{\phi(|a|)}C_1(a)-\frac{\phi(|a|)}{|a|}\frac 1R +O\Bigg( \frac{\prod_{p\mid a} \big( 1+\frac 1{p^{\frac 13}} \big)} {R^{\frac 32} \exp \big(c \frac{(\log R)^{\frac 35}}{(\log\log R)^{\frac 15}} \big)}  \Bigg);
$$
$$
\sum_{\substack{r\leq R \\ (r,a)=1}} \frac{\mu^2(r)\log r}{r\phi(r)} = \frac{|a|}{\phi(|a|)}C_1(a)\sum_{p\nmid a} \frac{\log p}{p^2-p+1}-\frac{\phi(|a|)}{|a|}\frac {\log R+1}R +O\Bigg( \frac{\prod_{p\mid a} \big( 1+\frac 1{p^{\frac 13}} \big)} {R^{\frac 32} \exp \big(c \frac{(\log R)^{\frac 35}}{(\log\log R)^{\frac 15}} \big)}  \Bigg).
$$
\end{lemma}

\begin{proof}

We only prove the first of these estimates. Write $\eta_c(R):= \exp \big(c (\log R)^{\frac 35}/(\log\log R)^{\frac 15} \big)$. We have the identity 
$$\sum_{\substack{r > R \\ (r,a)=1}} \frac{\mu^2(r)}{r\phi(r)} = \sum_{d \mid a^{\infty}} \frac{\lambda(d)}{ d \prod_{p^{\nu}\parallel d} (p-1)^{\nu} } \sum_{\substack{m > R/d }} \frac{\mu^2(m)}{m\phi(m)},$$
which combined with Lemma \ref{lemma sum over r pure} gives that for some $c>0$,
\begin{align*}
\sum_{\substack{r > R \\ (r,a)=1}} \frac{\mu^2(r)}{r\phi(r)} &= \frac {1}R \sum_{d \mid a^{\infty}} \frac{\lambda(d)}{  \prod_{p^{\nu}\parallel d} (p-1)^{\nu} } +O\Big( \sum_{\substack{d \mid a^{\infty} \\ d\leq R^{\frac 12}}}  \frac{ \prod_{p\mid d} \left(1+\frac 1p \right)}{d^{\frac 12}R^{\frac 32} \eta_c(R)} +\sum_{\substack{d \mid a^{\infty} \\ R^{\frac 12}<d\leq R}}  \frac{ \prod_{p\mid d} \left(1+\frac 1p \right)}{ d^{\frac 12}R^{\frac 32} }\Big) \\
&= \frac{\phi(|a|)}{|a|}\frac 1R +O\bigg(  \frac{\prod_{p\mid a} \big( 1+\frac 1{p^{\frac 13}} \big)}{R^{\frac 32}\eta_c(R)}\bigg).
\end{align*}
The proof follows from a straightforward computation.
\end{proof}

\begin{lemma}
\label{lemma sum over r with gcd}
There exists an absolute constant $c>0$ such that if $a\in \mathbb Z_{\neq 0}$ and $R\in \mathbb R_{\geq 9}$ are such that $a_R:= \prod_{\substack{p \mid a \\ p\leq R}}p \leq R/\log R$, then we have
$$
\sum_{\substack{r\leq R }} \frac{\mu^2(r) \mu((r,a))\phi((r,a))}{r\phi(r)} = C_1(a)-\frac {\epsilon_{a=\pm 1}}R +O\Bigg( \frac{a_R^{\frac 12} \prod_{p\mid a} \left( 1+ \frac 2{p^{\frac 13}} \right)} {R^{\frac 32} \exp \big(c \frac{(\log (R/a_R))^{\frac 35}}{(\log\log (R/a_R))^{\frac 15}} \big)}  \Bigg);
$$
$$
\sum_{\substack{r\leq R }} \frac{\mu^2(r) \mu((r,a))\phi((r,a))}{r\phi(r)} \log r = (\gamma-1)C_1(a)-\frac {\epsilon_{a=\pm 1}(\log R+1)}R +O\Bigg( \frac{a_R^{\frac 12} \prod_{p\mid a} \left( 1+ \frac 2{p^{\frac 13}} \right)} {R^{\frac 32} \exp \big(c \frac{(\log (R/a_R))^{\frac 35}}{(\log\log (R/a_R))^{\frac 15}} \big)}  \Bigg);
$$
\end{lemma}
\begin{proof}
The first of these estimates follows from writing
$$ \sum_{\substack{r\leq R }} \frac{\mu^2(r) \mu((r,a))\phi((r,a))}{r\phi(r)} = \sum_{d\mid a} \frac{\mu(d)}d \sum_{\substack{ m\leq R/d \\ (m,a)=1}} \frac{\mu^2(m) }{m\phi(m)},  $$
applying Lemma \ref{lemma sum over r coprime to a} and performing a straightforward calculation. 

\end{proof}

\section{The sum over all moduli}

\label{section all moduli}

In order to prove Theorem \ref{theorem main full range}, we need to understand the quantity $\rho_R^*(x;q,a)$ for more moderate values of $q$.

\begin{lemma}

\label{lemma whole range of q}

Uniformly for $0 < |a| \leq x^{\frac 12}$ and $R\leq x^{\frac 12}$ we have 
$$ \sum_{q\leq x} \rho_R^*(x;q,a) = x \sum_{r\leq R} \frac{\mu^2(r)\mu((r,a))\phi((r,a))}{r\phi(r)} \left(\log \frac{x}{r^2}+2\gamma-1\right) +O(Rx^{\frac 12}+|a|(\log x)^2).$$
\label{lemma average of rho q<x}
\end{lemma}

\begin{proof}

For those $q$ in the interval $(x^{\frac 12},x]$, we apply Lemma \ref{lemma big values of q} (i) to obtain that
\begin{multline*} \sum_{x^{\frac 12}<q\leq x} \rho_R^*(x;q,a) =  x\sum_{r\leq R}\frac{\mu^2(r)\mu((r,a))\phi((r,a))}{r\phi(r)} \sum_{t\leq x^{\frac 12}/r} \frac 1t \left( 1- \frac{t}{x^{\frac 12}/r} \right) \\+ O(Rx^{\frac 12} + |a|(\log x)^2). \end{multline*}

As for the remaining values of $q$, we take $y=a_+:=\max\{0,a\}$ in Lemma \ref{lemma:additive character sum2} and obtain 
\begin{align*} \sum_{q\leq x^{\frac 12}} \rho_R^*(x;q,a) 
&=  \sum_{q\leq x^{\frac 12}}\frac {x-a_+}q \sum_{\substack{r\leq R \\ r\mid q}} \frac{\mu^2(r)\mu((r,a))\phi((r,a))}{\phi(r)} +O(Rx^{\frac 12})\\
&= x\sum_{r\leq R}\frac{\mu^2(r)\mu((r,a))\phi((r,a))}{r\phi(r)} \sum_{t\leq x^{\frac 12}/r} \frac 1t +O(Rx^{\frac 12}+|a|(\log x)^2).
\end{align*}

The desired estimate follows from applying the standard estimate on the harmonic sum.

\end{proof}

In the following lemma we show that the average of $\psi^*(x;q,a)$ is very small when $(q,a)>1$. 

\begin{lemma}
\label{lemma psi negligible}
We have that

$$ \sum_{\substack{q\leq x \\ (q,a)>1}} \psi^*(x;q,a) \ll  x^{\frac 12}\log |a|. $$

\end{lemma}
\begin{proof}
We write 
\begin{align*}
\sum_{\substack{q\leq x \\ (q,a)>1}} \psi^*(x;q,a) &\leq  \sum_{\substack{q\leq x^{\frac 12} \\ (q,a)>1}} \psi^*(x;q,a) + \sum_{\substack{s\leq x^{\frac 12} -  ax^{-\frac 12} }} \sum_{\substack{ n\leq x \\ n\equiv a \bmod s \\ (n,a)>1 }} \Lambda(n) \\
&\ll  \sum_{q\leq x^{\frac 12}}\sum_{p^{\nu} \parallel a } \nu\log p+ \sum_{s\leq x^{\frac 12}} \sum_{p^{\nu} \parallel a } \nu\log p 
\leq  2 x^{\frac 12} \log |a|. 
\end{align*}

\end{proof}

We are now ready to estimate the average of $\psi^*(x;q,a)-\rho_R^*(x;q,a)$ over $q\leq x/M$. 

\begin{proposition}
\label{proposition unevaluated r sum}
Fix $A,B\geq 1$ and $0<\lambda <1/4$. We have for $0<|a|\leq x^{\lambda}$, $1\leq M \leq R \leq x^{\frac 12}$ and $M\leq (\log x)^A$ that

\begin{multline*} \sum_{q\leq \frac xM} (\psi^*(x;q,a)-\rho_R^*(x;q,a)) = x\Big[ C_1(a)\log x +C_1(a)+2C_2(a) \\- \sum_{r\leq R} \frac{\mu^2(r)\mu((r,a))\phi((r,a))}{r\phi(r)} \left(\log \frac{x}{r^2}+2\gamma-1\right)\Big] +O(Rx^{\frac 12})+O_{A,B,\lambda} \left( \frac x{(\log x)^B}\right).
\end{multline*}

\end{proposition}

\begin{proof}
Applying \cite[Proposition 6.1]{Fi} (which is based on the works \cite{Fo,BFI,FG,FGHM}) and the elementary estimate \cite[Lemma 5.2]{Fi} (see also \cite[Lemma 13.1]{FGHM}, we obtain that
\begin{multline} \sum_{\substack{q\leq \frac xM\\(q,a)=1}} \psi^*(x;q,a)= x\Big[C_1(a) \log x + C_1(a)+2C_2(a) -  \sum_{\substack{  s \leq M \\ (s,a)=1}} \frac{1}{\phi(s)} \left( 1-\frac sM\right)\Big] \\+O(2^{\omega(a)} M\log x)+O_{A,B,\lambda}\left( 2^{\omega(a)}\frac x{(\log x)^B}\right). \label{equation average of psi^* q<x/M}\end{multline}
Note that \cite[Proposition 6.1]{Fi} has the extra condition that $M$ should be an integer, however going through the proof we see that in general we have 
$$ \sum_{\substack{ s < M-\frac{aM}x \\ (s,a)=1 }} \frac 1{\phi(s)} \left( 1-\frac sM \right) = \sum_{\substack{ s \leq  M \\ (s,a)=1 }} \frac 1{\phi(s)} \left( 1-\frac sM \right)+O\left(\frac{|a|}x \right), $$
and hence this extra condition can be removed at the cost of an admissible error term. Moreover, by Lemma \ref{lemma psi negligible}, we can remove the condition $(q,a)=1$ at the cost of the error term $O(x^{\frac 12}\log x)$. Finally, we combine Lemmas \ref{lemma big values of q} (ii) and 
\ref{lemma average of rho q<x} to obtain that
\begin{multline} \sum_{q\leq \frac xM } \rho_R^*(x;q,a) = x\Big[ \sum_{r\leq R} \frac{\mu^2(r)\mu((r,a))\phi((r,a))}{r\phi(r)} \left(\log \frac{x}{r^2}+2\gamma-1\right) \\-\sum_{\substack{  s \leq M \\ (s,a)=1}} \frac{1}{\phi(s)} \left( 1-\frac sM\right)\Big]  +O(Rx^{\frac 12}+|a|(\log x)^2).
\label{equation sum over r not evaluated}   
\end{multline}

Subtracting this from \eqref{equation average of psi^* q<x/M} gives the desired result. 

\end{proof}

\begin{proof}[Proof of Theorem \ref{theorem main full range} (i)]
The result follows from combining Proposition \ref{proposition unevaluated r sum} with Lemma \ref{lemma sum over r with gcd}, and a straightforward calculation.
\end{proof}

\section{The coprimality condition}

\label{section coprimality}

In this section we prove Theorems \ref{theorem main dyadique} (ii) and \ref{theorem main full range} (ii). This amounts to controlling the contribution of $\rho_R^*(x;q,a)$ with $(q,a)>1$ (this is much easier for $\psi^*(x;q,a)$ and was already done in Lemma \ref{lemma psi negligible}). The condition $(q,a)=1$ is easier to treat than the condition $(q,a)>1$, and hence 
we will estimate sums over $(q,a)=1$ directly.

Theorem \ref{theorem main dyadique} (ii) will follow from the following lemma.

\begin{lemma}
Let $R,N\leq x$, and $|a|<x/N$ be such that $R\geq |a|N$. Then we have 
\begin{equation}
\sum_{\substack{\frac x{N}<q\leq x \\ (q,a)=1}} \rho_R^*(x;q,a) = x \sum_{\substack{ s \leq N \\ (s,a)=1}} \frac 1{\phi(s)}\left( 1-\frac sN \right)+O( 2^{\omega(a)}RN+|a|\log N).\label{equation gcd=1 dyadique}
\end{equation}
\label{lemma average of rho coprime} 
\end{lemma} 
\noindent (Compare with Lemma \ref{lemma big values of q} (ii).)
\begin{proof}

Following the proof of Lemma \ref{lemma big values of q}, we write 
\begin{align*}
\sum_{\substack{ \frac xN \leq q \leq x \\ (q,a)=1 }} \rho^*_R(x;q,a) &= \sum_{1\leq s < N-\frac{aN}x} \sum_{\substack{a+\frac{sx}N < n \leq x \\ n\equiv a \bmod s \\ (\frac{n-a}s,a)=1}} F_R(n).
\end{align*}
Applying M\"obius inversion and Lemma \ref{lemma:additive character sum2}, we see that the inner sum equals
\begin{align}
&\sum_{d\mid a} \mu(d)  \sum_{\substack{a+\frac{sx}N < n \leq x \\ n\equiv a \bmod ds }} F_R(n) =
 \frac{x-(a+sx/N)}s \sum_{d\mid a} \frac{\mu(d)}d \sum_{\substack{r\leq R \\ r\mid ds }} \frac{\mu^2(r)\mu((r,a))\phi((r,a))}{\phi(r)} + O(2^{\omega(a)}R).
\label{equation analogue of first lemma}
\end{align}

Hence,
\begin{multline}
\sum_{\substack{ \frac xN \leq q \leq x \\ (q,a)=1 }} \rho^*_R(x;q,a) =  x\sum_{1\leq s < N-\frac{aN}x} \frac 1s\left( 1-\frac sN -\frac ax\right)  \sum_{d\mid a} \frac{\mu(d)}d \sum_{\substack{r\leq R \\ r\mid ds}} \frac{\mu^2(r)\mu((r,a))\phi((r,a))}{\phi(r)} \\+ O(2^{\omega(a)}RN).
\label{equation average rho coprime without condition on M}
\end{multline}

Since $|a|N\leq R$, for $a>0$ the innermost sum equals
$$ \sum_{\substack{ r\mid ds}} \frac{\mu^2(r)\mu((r,a))\phi((r,a))}{\phi(r)} = \begin{cases}
\frac{ds}{\phi(ds)} & \text{ if } (ds,a)=1, \\
0 &\text{ otherwise.}
\end{cases} $$
If $a<0$, then we need to add an error term for the term $d=a$; this error term is easily seen to sum to $O(1)$. Therefore,
\begin{align*}
\sum_{\substack{ \frac xN \leq q \leq x \\ (q,a)=1 }} \rho^*_R(x;q,a) &= x \sum_{\substack{1\leq s < N-\frac{aN}x \\ (s,a)=1}} \frac 1s\left( 1-\frac sN -\frac ax\right)  \sum_{\substack{d\mid a \\ (d,a)=1 }} \frac{\mu(d)}d \frac{ds}{\phi(ds)} + O(2^{\omega(a)}RN). \\
&=x \sum_{\substack{ s \leq N \\ (s,a)=1}} \frac 1{\phi(s)}\left( 1-\frac sN \right)+ O(2^{\omega(a)}RN+|a|\log N).
\end{align*}
\end{proof}

\begin{proof}[Proof of Theorem \ref{theorem main dyadique} (ii)]
Combine Lemmas \ref{lemma average of psi}, \ref{lemma psi negligible} and \ref{lemma average of rho coprime}.

\end{proof}

In order to prove Theorem \ref{theorem main full range} (ii), we need to have an estimate on the sum of $\rho_R^*(x;q,a)$ over all $q\leq x$ coprime to $a$. We start with an elementary lemma.

\begin{lemma}
\label{lemma harmonic sum}
If $a\neq 0$ and $r\geq 1$ are integers, then for $y \in \mathbb R_{\geq 1/2}$ we have the estimate
$$ \sum_{\substack{ n\leq y  \\ (n,a)=1 \\ r \mid n}}  \frac 1n =  \delta_{(r,a)=1} \Big( \frac{\phi(a)}{ar} \Big(\log \frac yr + \gamma +\sum_{p\mid a} \frac {\log p}{p-1}\Big)+O\left(\frac{2^{\omega(a)}}{y}\right)\Big),$$
where $\delta_{(r,a)=1}$ equals $1$ when $(r,a)=1$, and is zero otherwise.
\end{lemma}
\begin{proof}

If $(r,a)>1$, then the sum on the left hand side is clearly zero. Otherwise, we apply M\"obius inversion and the standard estimate on the harmonic sum to obtain that
$$ \sum_{\substack{ n\leq y  \\ (n,a)=1 \\ r \mid n}}  \frac 1n= \frac 1r \sum_{d\mid a} \frac{\mu(d)}d \Big( \log  \frac y{rd} +\gamma+O\left( \frac {rd}y \right)\Big). $$
The proof follows from a standard calculation.
\end{proof}

The next lemma is an analogue of Lemma \ref{lemma average of rho q<x}.
\begin{lemma}
\label{lemma sum coprime rho full range}
For $0\neq |a| < x^{\frac 12}$ and $R\leq x^{\frac 12}$, the following holds:
\begin{multline*} \sum_{\substack{q\leq x \\ (q,a)=1}} \rho^*_R(x;q,a) = x \frac{\phi(|a|)}{|a|} \sum_{ \substack{ r\leq R \\ (r,a)=1} } \frac{\mu^2(r)}{r\phi(r)} \Big( \log \frac{x}{r^2} +2\gamma-1+\sum_{p\mid a} \frac{\log p}{p-1} \Big)\\+x\frac{\phi(|a|)}{|a|}\sum_{p\mid a} \frac{\log p}{p-1} \sum_{ \substack{t \leq R/p \\ (t,a)=1}} \frac{\mu^2(t) }{t\phi(t)}+O(|a| \log x +2^{\omega(a)} R x^{\frac 12} ).
\end{multline*}\end{lemma}

\begin{proof}
Arguing as in Lemma \ref{lemma average of rho q<x}, we cut the sum at $q=x^{\frac 12}$ and exchange divisors. Applying M\"obius inversion, setting $y=a_+:=\max \{0,a\}$ in Lemma \ref{lemma:additive character sum2} and applying \eqref{equation analogue of first lemma}, we compute
\begin{align*}
\sum_{\substack{q\leq x \\ (q,a)=1}} \rho_R^*(x;q,a) &= \sum_{\substack{q\leq x^{\frac 12} \\ (q,a)=1}} \sum_{\substack{ n\leq x \\ n \equiv a \bmod q \\ n> a}} F_R(n) + \sum_{\substack{1\leq s < x^{\frac 12}- a{x^{-\frac 12}} }} \sum_{\substack{ a+sx^{\frac 12}< n\leq x \\ n \equiv a \bmod s \\  (\frac{n-a}s,a)=1}} F_R(n) \\
&= \sum_{\substack{q\leq x^{\frac 12} \\ (q,a)=1}}  \frac {x-a_+}q \sum_{\substack{r\leq R \\ r\mid q}} \frac{\mu^2(r) \mu((r,a))\phi((r,a))}{\phi(r)} +O(Rx^{\frac 12})  \\&\hspace{-1cm}+ \sum_{\substack{s < x^{\frac 12}- a{x^{-\frac 12}} }}  \sum_{d\mid a} \frac{\mu(d)}d \frac {x-sx^{\frac 12} -a}s   \sum_{\substack{r\leq R \\ r\mid ds}} \frac{\mu^2(r) \mu((r,a))\phi((r,a))}{\phi(r)} +O(2^{\omega(a)}Rx^{\frac 12}) \\ &= I + II.
\end{align*}
To evaluate the first term, we apply Lemma \ref{lemma harmonic sum} and obtain that  
$$ I= x \frac{\phi(|a|)}{|a|} \sum_{ \substack{ r\leq R \\ (r,a)=1} } \frac{\mu^2(r)}{r\phi(r)} \Big( \log \frac{x^{\frac 12}}{r} +\gamma+\sum_{p\mid a} \frac{\log p}{p-1} \Big)+O(|a|\log x +2^{\omega(a)}R x^{\frac 12} ).$$

As for the second, we note that $r\mid ds$ if and only if $\tfrac r{(r,d)} \mid s$, and thus 
\begin{align*}
II &=  x  \sum_{r\leq R} \frac{\mu^2(r) \mu((r,a))\phi((r,a))}{\phi(r)} \sum_{d\mid a} \frac{\mu(d)}d \sum_{\substack{s < x^{\frac 12}- a{x^{-\frac 12}} \\ \frac r{(r,d)} \mid s}}   \frac{1-ax^{-1}-sx^{-\frac 12} }s +O(2^{\omega(a)}Rx^{\frac 12}) \\
&= x \sum_{r\leq R} \frac{\mu^2(r) \mu((r,a))\phi((r,a))}{\phi(r)} \sum_{d\mid a} \frac{\mu(d)}d  \frac{(r,d)}r \left[ \log \frac{x^{\frac 12}(r,d)}{r}+ \gamma-1 +O\left(\frac{|a|}x +\frac r{(r,d) x^{\frac 12}} \right) \right] \\
& \hspace{4cm} +O(2^{\omega(a)}Rx^{\frac 12}). 
 \end{align*}  
Note however that 
$$ \sum_{d\mid a} \frac{\mu(d) (r,d)}d= \delta_{(r,a)=1} \frac{\phi(|a|)}{|a|}; \hspace{1cm} \sum_{d\mid a} \frac{\mu(d) (r,d)}d \log (r,d)=\begin{cases}
-\frac{\phi(|a|)}{|a|} \frac{p\log p}{p-1} & \text{ if } (a,r)=p^k \\
0 & \text{ otherwise, }
\end{cases}$$
and hence
\begin{align*}
II &= x \frac{\phi(|a|)}{|a|} \sum_{\substack{r\leq R \\ (r,a)=1}} \frac{\mu^2(r) }{r\phi(r)} \left( \log \frac{x^{\frac 12}}{r}+ \gamma-1 \right)+x\frac{\phi(|a|)}{|a|}\sum_{p\mid a} \frac{\log p}{p-1} \sum_{ \substack{t \leq R/p \\ (t,a)=1}} \frac{\mu^2(t) }{t\phi(t)} \\&\hspace{2cm}+O( 2^{\omega(a)}R x^{\frac 12} ). 
\end{align*}  

\end{proof}

\begin{proposition}
\label{proposition coprime psi minus rho}
Fix $A,B\geq 1$ and $0<\lambda<1/4$. For $0\neq |a| \leq x^{\lambda}$, $1\leq M\leq R \leq x^{\frac 12}$ and $M\leq (\log x)^A$, we have that
\begin{multline*}
\sum_{\substack{q\leq \frac xM \\ (q,a)=1}} (\psi^*(x;q,a)-\rho_R^*(x;q,a)) = x \Big[ C_1(a) \log x + C_1(a)+ 2C_2(a)  \\-\frac{\phi(|a|)}{|a|} \sum_{ \substack{ r\leq R \\ (r,a)=1} } \frac{\mu^2(r)}{r\phi(r)} \Big( \log \frac{x}{r^2} +2\gamma-1+\sum_{p\mid a} \frac{\log p}{p-1} \Big) -\frac{\phi(|a|)}{|a|}\sum_{p\mid a} \frac{\log p}{p-1} \sum_{ \substack{t \leq R/p \\ (t,a)=1}} \frac{\mu^2(t) }{t\phi(t)} \Big] \\+O(2^{\omega(a)}Rx^{\frac 12} )+O_{A,B,\lambda}\left( 2^{\omega(a)}\frac x{(\log x)^B} \right).
\end{multline*}
\end{proposition}
\begin{proof}
Combining Lemmas \ref{lemma average of rho coprime} and \ref{lemma sum coprime rho full range} gives that
\begin{multline}
\sum_{\substack{q\leq \frac xM \\ (q,a)=1}} \rho_R^*(x;q,a) = x \frac{\phi(|a|)}{|a|} \sum_{ \substack{ r\leq R \\ (r,a)=1} } \frac{\mu^2(r)}{r\phi(r)} \Big( \log \frac{x}{r^2} +2\gamma-1+\sum_{p\mid a} \frac{\log p}{p-1} \Big) 
 \\ +x\frac{\phi(|a|)}{|a|}\sum_{p\mid a} \frac{\log p}{p-1} \sum_{ \substack{t \leq R/p \\ (t,a)=1}} \frac{\mu^2(t) }{t\phi(t)}-x \sum_{\substack{ s \leq M \\ (s,a)=1}} \frac 1{\phi(s)}\left( 1-\frac sM \right)+O(|a|\log x+2^{\omega(a)}R x^{\frac 12} ).
 \label{equation average rho coprime x/M}
\end{multline}
Applying \eqref{equation average of psi^* q<x/M} then yields the desired result.

\end{proof}

\begin{proof}[Proof of Theorem \ref{theorem main full range} (ii)]
The proof follows from combining Proposition \ref{proposition coprime psi minus rho} with Lemma \ref{lemma sum over r coprime to a}.
\end{proof}

\section{The quantity $\rho_R^*(x;q,a)$ when $(q,a)>1$}
\label{section not coprime}

Comparing Theorem \ref{theorem main full range} (i) and (ii), we see that the main terms agree when $a=\pm 1$ (since the sums on the left hand side coincide), but they are very different when $\omega(a)\geq 1$. More precisely, combining Lemmas \ref{lemma big values of q} (ii), \ref{lemma average of rho q<x}, \ref{lemma average of rho coprime} and \ref{lemma sum coprime rho full range} we see that for $0<|a|<x^{\frac 12}$ and $1\leq |a|N\leq R \leq x^{\frac 12}$,
\begin{align}
\sum_{ \substack{ q\leq \frac xN \\ (q,a)>1}} \rho_R^*(x;q,a) &= x \sum_{r\leq R} \frac{\mu^2(r)\mu((r,a))\phi((r,a))}{r\phi(r)} \left(\log \frac{x}{r^2}+2\gamma-1\right) \notag \\
&-x \frac{\phi(|a|)}{|a|} \sum_{ \substack{ r\leq R \\ (r,a)=1} } \frac{\mu^2(r)}{r\phi(r)} \Big( \log \frac{x}{r^2} +2\gamma-1+\sum_{p\mid a} \frac{\log p}{p-1} \Big)\notag \\&-x\frac{\phi(|a|)}{|a|}\sum_{p\mid a} \frac{\log p}{p-1} \sum_{ \substack{t \leq R/p \\ (t,a)=1}} \frac{\mu^2(t) }{t\phi(t)}+ O_a(Rx^{\frac 12} ). \label{equation average of rho not coprime}
\end{align}

It is not surprising that the main terms in this estimate are independent of $N$. Indeed applying Lemmas \ref{lemma big values of q} (ii) and \ref{lemma average of rho coprime} directly shows that 
$$  \sum_{ \substack{ \frac xN <  q \leq x \\ (q,a)>1}} \rho_R^*(x;q,a) \ll 2^{\omega(a)} RN +|a|\log N. $$

One can evaluate the sums in \eqref{equation average of rho not coprime} using Lemmas \ref{lemma sum over r coprime to a} and \ref{lemma sum over r with gcd}, resulting in the expression
\begin{multline*}
\sum_{ \substack{q \leq \frac xN \\ (q,a)>1}} \rho_R^*(x;q,a) =  \Big( \frac{\phi(|a|)}{|a|} \Big)^2 \frac xR \Big[  \log \frac x{R^2} +2\gamma-3 + \sum_{p\mid a} \frac{p+1}{p-1} \log p \Big]  \\
+O_{\epsilon,a}\left( \frac {x\log x}{R^{\frac 32} \exp \big(c \frac{(\log R)^{\frac 35}}{(\log\log R)^{\frac 15}} \big)} + Rx^{\frac 12} \right).
\end{multline*} 
Hence, the term $\rho_R^*(x;q,a)$ is on average of order $(N/R)(\log (x/R^2)+1)$. However the mass in this average is contained in the terms $q \ll_a x/R$, and thus it is more accurate to say that this term is of order $(\log (x/R^2)+1)$ on average for $q \ll_a x/R$, and is small for larger moduli.

In conclusion, while being quite small when $(q,a)>1$, the quantity $\rho^*_R(x;q,a)$ is not completely negligible and can be evaluated asymptotically on average over those values of $q$.

\section{Further proofs}

We will show in Lemma \ref{lemma discrepancy} that the total mass of $\rho^*_R(x;q,a)$ over all arithmetic progressions modulo $q$ is about $x$, and that this mass is concentrated in the invertible residue classes. It follows that $\rho^*_R(x;q,a)-x/\phi(q)$ is the approximate discrepancy of $\rho^*_R(x;q,a)$ in the invertible residue classes modulo $q$.

\begin{lemma}
\label{lemma discrepancy}
The total mass of $F_R(n)$ for $q<n\leq x$ in all residue classes modulo $q$ equals
$$ \sum_{1\leq a\leq q} \rho^*_R(x;q,a) = x-q +O(R), $$
and its mass in the invertible residue classes modulo $q$ is given by
$$ \sum_{\substack{1\leq a \leq q \\ (a,q)=1}} \rho^*_R(x;q,a) = x-q+ O \bigg( \frac xR \prod_{p\mid q} \left(2+\frac 1p \right) + 2^{\omega(q)} R \bigg). $$
\end{lemma}

\begin{proof}

The first estimate follows by a direct application of Lemma \ref{lemma:additive character sum2}:
\begin{align*} \sum_{a=1}^q \rho^*_R(x;q,a) =  \sum_{q<n\leq x }F_R(n)   &= (x-q) \sum_{\substack{r\leq R \\ r\mid 1}}    \frac{\mu^2(r)\mu((r,n))\phi((r,n))}{\phi(r)}+O(R) \\
&= x-q +O(R).
\end{align*}

To prove the second, we first use Möbius inversion, and then apply Lemma \ref{lemma:additive character sum2} with $a=0$. This gives the estimate
\begin{align*}
\sum_{\substack{1\leq a \leq q \\ (a,q)=1}} \rho^*_R(x;q,a) = \sum_{\substack{  q<n\leq x \\ (n,q)=1}} F_R(n) &= \sum_{d\mid q} \mu(d)  \sum_{\substack{q<n\leq x \\ n\equiv 0 \bmod d}} F_R(n) \\
& = (x-q)\sum_{d\mid q} \frac{\mu(d)}d \sum_{\substack{r\leq R \\ r\mid d}} \mu(r)+O(2^{\omega(q)} R)\\
& = (x-q)+(x-q)\sum_{\substack{d\mid q \\ d\neq 1}} \frac{\mu(d)}d \sum_{\substack{r\leq R \\ r\mid d}} \mu(r)+O(2^{\omega(q)} R).
\end{align*}
Now, writing $q':=\prod_{p\mid q} p$, we have that

\begin{align*}
\sum_{\substack{ d\mid q \\ d\neq 1}} \frac{\mu(d)}d \sum_{\substack{r\leq R \\ r\mid d}} \mu(r)&= \sum_{\substack{ d\mid q' \\ d\neq 1}} \frac{\mu(d)}d \sum_{\substack{r> R \\ r\mid d}} \mu(r) \ll \sum_{\substack{ r\mid q' \\ r>R}} \sum_{\substack{d\mid q' \\ r\mid d}} \frac {1}d \\&=  \prod_{p\mid q} \left( 1+\frac 1p \right) \sum_{\substack{ r\mid q' \\ r>R}} \frac 1r \prod_{p\mid r} \left( 1+\frac 1p \right)^{-1} \\
&\leq \frac 1R \prod_{p\mid q} \left( 1+\frac 1p \right)\sum_{\substack{ r\mid q'}}  \prod_{p\mid r} \left( 1+\frac 1p \right)^{-1}.
\end{align*}
The proof follows by multiplicativity.

\end{proof}

We now come back to the discrepancies of $\rho_R^*(x;q,a)$ in arithmetic progressions.

\begin{proof}[Proof of Proposition \ref{proposition discrepancies rho}]
Combining Lemmas \ref{lemma average of rho coprime} and \ref{lemma sum coprime rho full range} with Lemma \ref{lemma sum over r coprime to a} and \cite[Lemma 5.9]{Fi}\footnote{The exponent $205/538$ in this estimate can be replaced with $171/448$ thanks to Bourgain's result \cite{Bo}.} gives that
\begin{multline*}
\sum_{\substack{q\leq \frac xM \\ (q,a)=1}} \rho_R^*(x;q,a) = x \Big(C_1(a)\log \frac xM +C_1(a)+C_2(a) +\frac{\phi(|a|)}{|a|}\frac{\mu(a,M)}M\Big)\\-\left(\frac{\phi(|a|)}{|a|}\right)^2 \frac{x}R\Big(\log \frac x{R^2} +2\gamma-3+\sum_{p\mid a} \frac{p+1}{p-1}\log p\Big)+O_{a,\epsilon}\left(Rx^{\frac 12}+ \frac x{M^{\frac{619}{448}-\epsilon}}+\frac {x\log x}{R^{\frac 32} \exp \big(c \frac{(\log R)^{\frac 35}}{(\log\log R)^{\frac 15}} \big)} \right).
\end{multline*}
The result follows from subtracting the following classical elementary estimate (see for instance \cite[Lemma 13.1]{FGHM}, in which we can replace $\tau(a)$ by $2^{\omega(a)}$):
$$ \sum_{\substack{ q\leq \frac xM \\ (q,a)=1}} \frac x{\phi(q)}  = x\left[C_1(a) \log \frac xM +C_1(a)+C_2(a) +O\left( 2^{\omega(a)} M \frac{\log x}x\right)\right].$$

\end{proof}

\section*{Acknowledgements}
I would like to thank Robert C. Vaughan for introducing me to his approximation, as well as Régis de la Bretèche and James Maynard for fruitful conversations. This work was partly accomplished while the author was at the University of Michigan and at Universit\'e Paris Diderot, and was supported by a Postdoctoral Fellowship from the Fondation Sciences Math\'ematiques de Paris and a Discovery Grant from the NSERC. 

\appendix

\end{document}